\documentclass[envcountsame, orivec]{llncs}

\usepackage[utf8]{inputenc}
\usepackage[T1]{fontenc}
\usepackage{lmodern}
\usepackage[english]{babel}

\usepackage{amsmath}
\usepackage{amsfonts}
\usepackage{amssymb}

\usepackage{MnSymbol}
\usepackage{mathrsfs}
\usepackage{stmaryrd}

\usepackage{xcolor}
\usepackage{calc}
\usepackage{titletoc}
\usepackage{enumerate}
\usepackage{hyperref}
    \hypersetup{colorlinks=true,linkcolor=black}
    
\colorlet{TeinteP1}{black}
\colorlet{TeinteP2}{black!80}
\colorlet{TeinteP3}{black!60}

\makeatletter
\newlength {\@Maths@Op@LargLR}
\newcommand{\@Maths@OpenCloseExt}[3]{%
    \mathopen{%
        \left#1\vphantom{#2}\right.%
        \hspace{-\@Maths@Op@LargLR}%
    }%
    #2%
    \mathclose{%
        \hspace{-\@Maths@Op@LargLR}%
        \left.\vphantom{#2}\right#3%
    }%
}
\newcommand{\@Maths@OpenRelCloseExt}[5]{%
    \mathopen{%
        \left#1\vphantom{#2#4}\right.%
        \hspace{-\@Maths@Op@LargLR}%
    }%
    #2%
    \mathrel{%
        \hspace{-\@Maths@Op@LargLR}%
        \left.\vphantom{#2#4}\middle#3\right.%
        \hspace{-\@Maths@Op@LargLR}%
    }%
    #4%
    \mathclose{%
        \hspace{-\@Maths@Op@LargLR}%
        \left.\vphantom{#2#4}\right#5%
    }%
}

\newcommand{\enstq}[2]{%
    \@Maths@OpenRelCloseExt{\lbrace}{#1}{\mid}{#2}{\rbrace}%
}

\newcommand{\card}[1]{%
    \@Maths@OpenCloseExt{\lvert}{#1}{\rvert}%
}
\makeatother

\newcommand{\F}{\ensuremath{\mathalpha{\mathbb{F}}}}

\DeclareMathOperator{\parties}{\mathscr{P}}
\DeclareMathOperator{\rank}{rank}
\DeclareMathOperator{\spn}{\mathit{span}}%

\begin{document}
\title{A Corollary of Hamada and Ohmori's on Group Law over BIBD}%
\author{Arnaud Bannier\inst{1} \and Johann Barbier\inst{2} \and Eric Filiol\inst{1} \and Pierre Castel\inst{2}}%
\authorrunning{A. Bannier, J. Barbier, E. Filiol and P. Castel}%
\institute{ESIEA - CVO Lab. F-53 000 Laval FRANCE \and
    ARX Arceo - Security Lab. - F-35 580 Guichen FRANCE \\
    \email{arnaud.bannier@esiea-ouest.fr}}%
\maketitle
\begin{abstract}
In this note, we present an interesting corollary of a theorem of Hamada and Ohmori.
We prove that the complementary of PG$(n,2)$ is the only design, up to an isomorphism, whose blocks form a group for the symmetric difference.

\end{abstract}%

\section*{Introduction}

The construction and the existence of BIBDs of given $(v, k, \lambda)$ parameters is an open problem \cite[pp 36--57]{colbourn10handbook}.
Different works exhibit necessary conditions on these parameters (for example, the well--known Fisher's inequality \cite{fisher40examination}), but without any additional hypothesis, finding sufficient conditions is difficult.
A more natural approach consists in adapting algebraic structures to find infinite families of BIBDs.
In this way, properties of combinatorial designs can be equivalent to algebraic ones. For instance, Hadamard matrices are known to be equivalent to $(4n - 1, 2n - 1, n - 1)$--BIBDs \cite{beth99design} and so provide families of BIBDs \cite{seberry92hadamard}.
One of the most widespread strategies relies on group structures \cite{beth99design}.

A first approach consists in fixing an automorphism group to deduce associated designs \cite{kramer76tdesigns}.
Another one relies on finding a subset of a finite group with fixed properties called a \emph{difference set}. The elements of the group are the points of the design, the subset and its translates are the blocks of the design \cite{jungnickel92difference}\cite[chapter 6]{beth99design}. It is also natural to consider design blocks as elements of a finite group.

Some group laws on blocks of a design have already been highlighted in the literature.
In his paper \cite{kantor75symplectic}, Kantor gave one of the most relevant examples.
A symmetric BIBD satisfies the \emph{symmetric difference property} (SDP) if for any three blocks $B$, $C$, $D$ then $B \Delta C \Delta D$ is either a block or the complement of a block (where $\Delta$ denote the symmetric difference). Kantor investigated such designs and proved \cite[Theorem 3]{kantor75symplectic} that a group law can be defined on their blocks as follows. Choose a block $B$ (the neutral element).
For all blocks $X$ and $Y$, define $X + Y$ as $B \Delta X \Delta Y$ or its complement depending on which of the two results is a block.
With this addition, the blocks form an elementary abelian 2--group. Kantor also gave the family $\mathscr{S}^\epsilon(2m)$ of SDP--designs and proved that any other family of SDP--designs must have the same parameters.

Another example is provided by Kimberley \cite{kimberley71construction}. Let $(X,\mathcal{B})$ be a Hadamard 3--design. Then, the complement of any block is another block. A block $B$ is said to be \emph{good} if for all block $C$ distinct from $B$ and its complement $\overline{B}$, the symmetric difference $B \Delta C$ is again a block.
Remark that the original definition of Kimberley is different but equivalent to this one for the Hadamard 3--designs.
If $B$ is a good block, then $\overline{B}$ is also a good block and $\{B,\overline{B}\}$ is a \emph{good block class}.
Let $\mathscr{G}$ be the set of all good block classes and let $\mathscr{H}$ be the set $\mathscr{G} \cup \{\{X,\emptyset\}\}$.
Define the binary operation $\circ$ on $\mathscr{H}$ by $\{B,\overline{B}\} \circ \{C,\overline{C}\} = \{B \Delta C, \overline{B \Delta C}\}$. Then, $\mathscr{H}$ is an elementary abelian 2--group under the operation $\circ$, see \cite[Lemma 4.8]{kimberley71construction}.

Here, we propose an original approach similar with \cite{kantor75symplectic,kimberley71construction} which provides BIBDs whose blocks form a group for the symmetric difference. We show necessary conditions on parameters to provide BIBDs with such a structure and prove that they are sufficient with a result of Hamada and Ohmori. We emphasize that all such designs have been found.

Initially, the additional group structure allowed us to construct these designs and prove their uniqueness up to an isomorphism. This construction was in fact equivalent to that of Sylvester \cite{sylvester1867thoughts}. Then, we found that our main result can be seen as a corollary of a theorem of Hamada and Ohmori \cite[Theorem 4.2]{hamada75bibd}. It is in this perspective that we present this paper.

The paper is organized as follows: first, we recall basic definitions and notations. Then, we study the necessary conditions to provide the BIBDs with a group structure. We conclude using the result Hamada and Ohmori.

\section{Basic Notations and Definitions}\label{SecBasic}%

\begin{definition}
Let $X$ be a finite set with $v$ elements called \emph{points} and $\mathcal{B}$ be a set of subsets of $X$ called \emph{blocks}. The pair $(X,\mathcal{B})$ is said to be a \emph{simple $(v,k,\lambda)$ balanced incomplete block design} (or a simple $(v,k,\lambda)$--BIBD for short) if all the blocks contain exactly $k$ points and if every pair of distinct points is contained in exactly $\lambda$ blocks. It is also required that $v > k \geq 2$. Such a design is said \emph{symmetric}, and denoted SBIBD, if the number of points is equal to the number of blocks, or equivalently, if the cardinal of the intersection of two blocks is constant \cite[Corollary II.3.3]{beth99design}.
\end{definition}

It is well known that the equality $\lambda(v - 1) = k(k - 1)$ holds in any $(v,k,\lambda)$--SBIBD \cite[Definition II.3.1]{beth99design}.
Let $(X,\mathcal{B})$ be a simple $(v,k,\lambda)$--BIBD. Write $X = \{x_1,\ldots,x_v\}$ and $\mathcal{B} = \{B_1,\ldots,B_b\}$. The \emph{incidence matrix} of $(X,\mathcal{B})$ (for this order) is the matrix $M = (m_{i,j})^{1 \leq i \leq v}_{1 \leq j \leq b}$ defined by
\[
    m_{i,j} = \begin{cases}
        1 & \text{if } x_i \in B_j \enspace,\\
        0 & \text{if } x_i \nin B_j \enspace.
    \end{cases}
\]

\section{Necessary and Sufficient Conditions}

Let $(X,\mathcal{B})$ be a simple $(v,k,\lambda)$--BIBD.
The goal of this paper is to endow $\mathcal{B}$ with a group law.
As $\mathcal{B}$ is included in the power set $\parties(X)$ of $X$, choosing a binary operation $\ast$ such that $(\mathcal{B},\ast)$ is a subgroup of $(\parties(X),\ast)$ is natural.
Among all classical binary operations\footnote{By classical, we mean $\cup$, $\cap$, $\Delta$ and $\setminus$.} on $\parties(X)$, the symmetric difference $\Delta$ is the only one which gives rise to a group.
The empty set $\emptyset$, the neutral element of $(\parties(X),\Delta)$, is never in $\mathcal{B}$.
Thus, we define\footnote{We should define $\mathscr{B} = (\mathcal{B} \cup \{\emptyset\}, \tilde{\Delta})$ where $\tilde{\Delta}$ is the restriction of $\Delta$ on $\enstq{(B,C) \in (\mathcal{B} \cup \{\emptyset\})^2}{B \Delta C \in \mathcal{B} \cup \{\emptyset\}}$.} $\mathscr{B} = (\mathcal{B} \cup \{\emptyset\}, \Delta)$.
The purpose of this section is to study the conditions over the parameters $(v,k,\lambda)$ such that $\mathscr{B}$ is a group.

\begin{lemma}\label{LemParameters}%
Assuming that $\mathscr{B}$ is a group, the design $(X,\mathcal{B})$ is a symmetric $(4\lambda - 1, 2\lambda, \lambda)$-BIBD.
\end{lemma}

\begin{proof}
Let $B$, $B^\prime$ be two distinct elements of $\mathcal{B}$. By definition, each bock of $\mathcal{B}$ contains exactly $k$ points. As $B \Delta B^\prime$ is an element of $\mathcal{B}$, $\card{B \Delta B^\prime} = k$. Since
\[
    \card{B \Delta B^\prime} = \card{B} + \card{B^\prime} - 2 \card{B \cap B^\prime} = 2k - 2\card{B \cap B^\prime} \enspace ,
\]
the equality $2\card{B \cap B^\prime} = k$ holds. Noting that the cardinal of the intersection of two blocks is constant, $(X,\mathcal{B})$ is a symmetric BIBD and $\card{B \cap B^\prime} = \lambda$. Thus, we have $2\lambda = k$. From the equality $\lambda(v-1) = k(k-1)$, it follows that
\[
    v = \frac{k(k-1)}{\lambda} + 1 = \frac{2\lambda(2\lambda-1)}{\lambda} + 1 = 4\lambda - 1 \enspace .
\]
Hence, $(X,\mathcal{B})$ is a $(4\lambda - 1, 2\lambda, \lambda)$-SBIBD.\qed
\end{proof}

Write $X = \{x_1,\ldots,x_v\}$, $\mathcal{B} = \{B_1,\ldots,B_b\}$ and let $M$ be the incidence matrice of $(X,\mathcal{B})$. Let $C_1,\ldots,C_b$ be its columns, seen as elements of $\F_2^b$ and define $C_0 = (0,\ldots,0) \in \F_2^b$.
The blocks $B_i$ and $B_j$ ($1 \leq i,j \leq b$, $i \neq j$) are represented respectively by $C_i$ and $C_j$. It is easily seen that $B_i \Delta B_j$ is represented by $C_i + C_j$ and $B_i \Delta B_i = \emptyset$ by $C_0$.
Define $\mathscr{C} = (\{C_0,C_1,\ldots,C_b\}, +)$.
Consequently, $\mathscr{B}$ is a group if, and only if $\mathscr{C}$ is a group.
If this is the case, $\mathcal{C}$ has an additional structure of $\F_2$-vector space.

\begin{lemma}\label{LemRank}%
Let $n$ denote the rank of $M$.
Then $\mathscr{B}$ is a group if, and only if $(X,\mathcal{B})$ is a $(2^n - 1, 2^{n-1}, 2^{n-2})$-SBIBD.
\end{lemma}

\begin{proof}
By definition, $n$ is the dimension of the subspace $\spn(\mathscr{C})$ of $\F_2^{2^n - 1}$.
Assume that $\mathscr{B}$ is a group.
It follows that $\mathscr{C}$ equals $\spn(\mathscr{C})$, so $\mathscr{C}$ has $2^n$ elements and the number $b$ of blocks equals $2^n - 1$.
From Lemma \ref{LemParameters}, we have $v = b = 4\lambda - 1$ and $k = 2\lambda$. Then $2^n - 1 = 4\lambda - 1$, that is $\lambda = 2^{n-2}$. Consequently, $(X,\mathcal{B})$ is a $(2^n - 1, 2^{n-1}, 2^{n-2})$-SBIBD.

Conversely, assume that $(X,\mathcal{B})$ is a $(2^n - 1, 2^{n-1}, 2^{n-2})$-SBIBD. Of course, $\mathscr{C} \subset \spn(\mathscr{C})$ and $\card{\mathscr{C}} = \card{\spn(\mathscr{C})} = 2^n$. It follows that $\mathscr{C}$ is a vector space and $\mathscr{B}$ is a group.
\qed
\end{proof}

Let us now consider the following theorem due to Hamada and Ohmori \cite[Theorem 4.2]{hamada75bibd}.

\begin{theorem}\label{ThmHamada}%
Let $D$ be a $(2^n - 1, 2^{n-1}, 2^{n-2})$-SBIBD and let $N$ be an incidence matrix of $D$. Then,
\[
    \rank_2(M) \geq n
\]
and the equality is attained when and only when the design $D$ is isomorphic with the complementary design of PG$(n - 1, 2)$.
\end{theorem}

Now, we can state the main result of this note, which is a direct consequence of Lemma \ref{LemRank} and Theorem \ref{ThmHamada}.

\begin{corollary}
Let $n$ denote the rank of $M$.
Then the complementary of PG$(n - 1, 2)$ is the only BIBD up to isomorphism such that $\mathscr{B}$ is a group.
\end{corollary}

\section{Conclusion}

In this paper, we have considered combinatorial designs provided with a group law inherent to the blocks, namely the symmetric difference. 
We began by reviewing necessary conditions and found that these designs must have the parameters $(2^n - 1,2^{n-1},2^{n-2})$. 
Using a result due to Hamada and Ohmori, we proved that the complementary of PG$(n - 1, 2)$ is the only one with this property, up to an isomorphism.

\bibliographystyle{amsplain}
\bibliography{Corollary_of_Hamada_and_Ohmori}

\end{document}